\documentclass{amsart}
\usepackage{graphicx}
\newtheorem{thm}{Theorem}[section]

\theoremstyle{definition}

\theoremstyle{remark}
\newtheorem{rem}[thm]{Remark}
\numberwithin{equation}{section}

\begin{document}

\title[Gamma and polygamma functions]{“Inequalities involving the gamma and digamma functions}%
\author{necdet batir}%
\address{department of  mathematics, nev{\c{s}}ehir hbv university, nev{\c{s}}ehir, 50300 turkey}%
\email{nbatir@hotmail.com}%


\date{March 23,2018}%
\begin{abstract}
We improve the upper bounds of the following inequalities proved in [H. Alzer and N. Batir, Monotonicity properties of the gamma function, Appl. Math. Letters,
20(2007), 778-781].
\begin{equation*}
exp\left(-\frac{1}{2}\psi\left(x+\frac{1}{3}\right)\right)<\frac{\Gamma(x)}{\sqrt{2\pi}x^xe^{-x}}<exp\left(-\frac{1}{2}\psi(x)\right),
\end{equation*}
and
$$
\frac{1}{2}\psi'(x+1/3))<\log x-\psi(x)<\frac{1}{2}\psi'(x).
$$
Here $\Gamma$ is the classical gamma function and $\psi$ is the digamma function.
\end{abstract}
\maketitle
\section{Introduction} As it is well known the gamma function is defined by the improper integral
\begin{equation*}
\Gamma(z)=\int_{0}^{\infty}t^{z-1}e^{-t}dt
\end{equation*}
for all complex numbers $z$ except the non-positive integers, and privides an extention of  the factorial function. It is well known that it is one of the most important special functions and has very important applications in probability theory, combinatorics, statistical and quantum mechanics, number theory, and nuclear physics. It satisfies the fundamental functional equation $\Gamma(z+1)=z\Gamma(z)$. The Weierstrass' product form of it is given by
\begin{equation}\label{e:1}
\frac{1}{\Gamma(z)}=ze^{\gamma z}\prod_{n=1}^{\infty}\left(1+\frac{z}{n}\right)e^{-z/n},
\end{equation}
where
$$
\gamma=\lim\limits_{n\to\infty}\left(1+\frac{1}{2}+\frac{1}{3}+\cdots+\frac{1}{n}-\log n\right)=0.57721...
$$
is the Euler-Mascheroni constant, see \cite[pg. 346]{9}. The most important function related to the gamma function is the digamma or psi functin, which is defined by the logarithmic derivative of $\Gamma$, that is, $\psi(x)=\Gamma'(x)/\Gamma(x)$, $x>0$. Furthermore, the derivatives $\psi', \psi'',...,$ are called polygamma functions. The digamma and polygamma functions are also very important special functins and they have important applications in mathematics and other disciplines such as physics and statistics. They are also connected to  generalized harmonic numbers and  many other special functions such as the Riemann zeta, Hurwitz zeta and Clausen functions.  Over the decades, many mathematicians studied these functions and they obtained remarkable inequalities, interesting properties including monotonicity and convexity, please see [2, 4, 7, 8, 12-22] and references therein.  Taking the logarithm of both sides of  identity (\ref{e:1}), we obtain for $x>-1$
\begin{equation}\label{e:2}
\log \Gamma(x+1)=-\gamma x+\sum_{k=1}^{\infty}\bigg[\frac{x}{k}-\log (x+k)+\log k\bigg].  
\end{equation}
Differentiation gives
\begin{equation}\label{e:3}
\psi(x+1)=-\gamma+\sum_{k=1}^{\infty}\bigg[\frac{1}{k}-\frac{1}{k+x}\bigg].
\end{equation}

In \cite{3} the authors proved that the function
$$
G_c(x)=\log \Gamma(x)-x\log x+x-\frac{1}{2}\log (2\pi)+\frac{1}{2}\psi(x+c)
$$
is completely monotonic on $(0,\infty)$ if and only if $c\geq1/3$; while $-G_b(x)$ is completely monotonic on $(0,\infty)$ if and only if $b=0$. This result has many conclusions. For example,
$G_{1/3}(x)>0$ and $G_0(x)<0$ lead to
\begin{equation}\label{e:4}
exp\left(-\frac{1}{2}\psi\left(x+\frac{1}{3}\right)\right)<\frac{\Gamma(x)}{\sqrt{2\pi}x^xe^{-x}}<exp\left(-\frac{1}{2}\psi(x)\right).
\end{equation}
Similarly, $G_{1/3}'(x)<0$ and $G_0'(x)>0$ lead to
\begin{equation}\label{e:5}
\frac{1}{2}\psi'(x+1/3)<\log x-\psi(x)<\frac{1}{2}\psi'(x).
\end{equation}
In 2011 C. Mortici \cite{11} improved the upper and lower bounds given in (\ref{e:4}) and proved the following inequalities for $x\geq2$:
\begin{align}\label{e:6}
&exp\left(-\frac{1}{2}\psi\left(x+\frac{1}{3}\right)+\frac{1}{72x^2}\right)<\frac{\Gamma(x)}{\sqrt{2\pi}x^xe^{-x}}\notag\\
&<exp\left(-\frac{1}{2}\psi\left(x+\frac{1}{3}\right)+\frac{1}{72x^2}+\frac{11}{3240x^3}\right).
\end{align}
The lower bounds in (\ref{e:4}) and (\ref{e:5}) are very accurate but the same thing is not valid for the upper bounds. In the new paper \cite{5} the author improved the upper bound given in (\ref{e:4}) and proved for $x>0$ that
\begin{equation}\label{e:7}
exp\left(-\frac{1}{2}\psi(\delta_*(x))\right)<\frac{\Gamma(x)}{\sqrt{2\pi}x^xe^{-x}}<exp\left(-\frac{1}{2}\psi(\delta^*(x))\right),
\end{equation}
where
$$
\delta_*(x)=x+\frac{1}{3}\quad \mbox{and} \quad \delta^*(x)=\frac{1}{2}\frac{1}{(x+1)\log(1+1/x)-1}.
$$
In the same paper the author established the following inequalities
\begin{equation}\label{e:8}
exp\left(x\psi\left(\frac{x}{\log (x+1)}\right)\right)<\Gamma(x+1)<exp\left(x\psi\left(\frac{x}{2}+1\right)\right).
\end{equation}
The function $x \to \log(x)-\psi(x)$ in (\ref{e:5}) has attracted the attentions of  many mathematicians and they have offered bounds of different forms for it. For example, in \cite{10} the authors developed the following inequality for $x>1$:
\begin{equation}\label{e:9}
\frac{1}{2x}<\log x-\psi(x)<\frac{1}{x}.
\end{equation}
In \cite{1} the author extended it to $x>0$. Refinements of (\ref{e:9}) were given in \cite[Thm. 5]{6} as following:
\begin{equation*}
\frac{1}{2x}+\frac{1}{12(x+1/4)^2}<\log x-\psi(x)<\frac{1}{2x}+\frac{1}{12x^2}
\end{equation*}
and
\begin{equation*}
\frac{1}{2x}+\frac{1}{12x^2}-\frac{1}{12x^4}<\log x-\psi(x)<\frac{1}{2x}+\frac{1}{12x^2}-\frac{1}{120(x+1/8)^4}.
\end{equation*}
Our aim in this work is to improve the upper bounds given in (\ref{e:4}) and (\ref{e:5}). We also offer simplier lower bound in (\ref{e:8}).

In the proofs we need the following identities. Taking logarithm of both sides of identity (\ref{e:1}), we obtain for $x>-1$
\begin{equation}\label{e:10}
\log\Gamma(x+1)=-\gamma x+\sum\limits_{k=1}^\infty\left[\frac{x}{k}-\log(x+k)+\log k\right].
\end{equation}
Differentiation gives
\begin{equation}\label{e:11}
\psi(x+1)=-\gamma+\sum\limits_{k=1}^\infty\left[\frac{1}{k}-\frac{1}{k+x}\right] \quad x>-1.
\end{equation}
Differentiation successively gives for $n\in\mathbb{N}$ and $z\in\mathbb{C}$ , which is not a negative integer,
\begin{equation*}
\psi^{(n)}(z)=(-1)^{n-1}n!\sum_{k=0}^{\infty}\frac{1}{(z+k)^{n+1}}.
\end{equation*}
\section{Main results}
\begin{thm}For $x>0$ the following inequalties hold:
\begin{equation}\label{e:12}
\frac{1}{2}\psi'(\alpha(x))<\log x-\psi(x)<\frac{1}{2}\psi'(\beta(x)),
\end{equation}
\end{thm}
where
$$
\alpha(x)=x+\frac{1}{3}\quad \mbox{and}\quad \beta(x)=\frac{1}{\sqrt{2/x-2\log(1+1/x)}}.
$$
\begin{proof}If we use the functional equation $\Gamma(x+1)=x\Gamma(x)$, we get
\begin{equation*}
  \log x=\log \Gamma(x+1)-\log \Gamma(x).
\end{equation*}
Hence, if we employ (\ref{e:2}), we obtain
\begin{equation}\label{e:13}
\log x=-\gamma+\sum_{k=1}^{\infty}\bigg[\frac{1}{k}-\log (x+k)+\log (x+k-1)\bigg].
\end{equation}
From Taylor's theorem it follows that
\begin{equation}\label{e:14}
\log (x+k)-\log (x+k-1)=\frac{1}{x+k-1}-\frac{1}{2(k+\tau(k))^2},
\end{equation}
where $x-1<\tau(k)<x$ . Thus by the help of (\ref{e:14}), (\ref{e:13}) can be written as following:
\begin{equation}\label{e:15}
\log x=-\gamma+\sum_{k=1}^{\infty}\bigg[\frac{1}{k}-\frac{1}{x+k-1}+\frac{1}{2(k+\tau(k))^2}\bigg].
\end{equation}
Taking into account (\ref{e:11}), (\ref{e:15}) becomes
\begin{equation}\label{e:16}
\log x-\psi(x)=\frac{1}{2}\sum_{k=1}^{\infty}\frac{1}{(k+\tau(k))^2}.
\end{equation}
From (\ref{e:14}) we get
\begin{equation}\label{e:17}
\tau(k)=\bigg[\frac{2}{x+k-1}-2\log\bigg(1+\frac{1}{x+k-1}\bigg)\bigg]^{-\frac{1}{2}}-k.
\end{equation}
We shall show that $\tau$ is strictly increasing on $(1,\infty)$. For  this purpose we define
$$
f(u)=\bigg[2\bigg(\frac{1}{u}-\log\bigg(1+\frac{1}{u}\bigg)\bigg)\bigg]^{-\frac{1}{2}}-u.
$$
Differentiating $f$, after replacing $u$ by $1/t$, we get
$$
-\frac{1}{t^2}f'(t)=-\frac{t}{t+1}2^{-3/2}(h(t))^{-3/2}+\frac{1}{t^2},
$$
where $h(t)=t-\log(t+1)$. Therefore in order to show that  $\tau$ is strictly increasing, we only need to see
$$
\theta(t):=h(t)-\frac{t^2}{2(t+1)^{2/3}}<0.
$$
Differentiation gives \begin{align*}
\theta'(t)=&\frac{t}{t+1}-\frac{t(t+1)^{2/3}-t^2(t+1)^{-1/3}/3}{(t+1)^{4/3}}\\
&=\frac{t}{(t+1)^{5/3}}\bigg[(t+1)^{2/3}-1-\frac{2t}{3}\bigg]<0
\end{align*}
by the well known Bernoulli inequality $(1+x)^\delta<1+\delta x,\,0<\delta<1$ and $t>0$. Thus, we have  $\theta(t)>\theta(0)=0$, that is, $f$ is strictly increasing on $(0,\infty)$. Hence we conclude from (\ref{e:17}) that
\begin{equation}\label{e:18}
\frac{1}{2}\sum_{k=1}^{\infty}\frac{1}{(k+\tau(\infty))^2}<\log x-\psi(x)<\frac{1}{2}\sum_{k=1}^{\infty}\frac{1}{(k+\tau(1))^2}.
\end{equation}
From  (\ref{e:14}) we have
\begin{equation}\label{e:19}
\tau(1)=\bigg[\frac{2}{x}-2\log(1+1/x)\bigg]^{-\frac{1}{2}}-1.
\end{equation}
Applying L'Hospital rule it is not difficult to see that
\begin{equation}\label{e:20}
\tau(\infty)=\lim\limits_{k\to\infty}\tau(k)=\lim\limits_{t\to 0}f(1/t)+x-1=x+\frac{1}{3}-1.
\end{equation}
Using (\ref{e:19}) and (\ref{e:20}) the proof follows from (\ref{e:18}).
\end{proof}
The upper bound in (\ref{e:7}) is very accurate but it is not very useful in practise because of it structure. The following theorem gives much simpler upper bound, which clearly improves the upper bound of (\ref{e:12}), and has the advantages of simplicity.
\begin{thm} For $x\geq0$ we have
\begin{equation}\label{e:21}
\frac{1}{2}\psi'\left(x+\frac{1}{3}\right)<\log x-\psi(x)<\frac{1}{2}\psi'\left(x+\frac{1}{3}-\frac{1}{12x+3}\right).
\end{equation}
\end{thm}
\begin{proof} Since $\psi'$ is strictly decreasing on $(0,\infty)$, in the light of  Theorem 2.1, it suffices to show for $x>0$ that
\begin{equation}\label{e:22}
\frac{1}{\sqrt{\frac{2}{x}-2\log(1+1/x)}}>x+\frac{1}{3}-\frac{1}{12x+3}
\end{equation}
or simplifying
$$
H(x):=\log(1+1/x)-\frac{1}{x}+\frac{1}{2\left(x+\frac{1}{3}-\frac{1}{12x+3}\right)^2}>0.
$$
Differentiation gives
$$
H'(x)=-\frac{63+188x}{x^3(x+1)(12x+7)^3},
$$
so that $H$ is strictly decreasing on $(0,\infty)$. This concludes with $H(x)>\lim_{x\to\infty}H(x)=0$.
\end{proof}
Our next theorem improves the upper bound given in (\ref{e:4}).
\begin{thm} For $x>0$ we have
\begin{equation*}
\exp\left(-\frac{1}{2}\psi\left(x+\frac{1}{3}\right)\right)<\frac{\Gamma(x)}{\sqrt{2\pi}x^xe^{-x}}<\exp\left(-\frac{1}{2}\psi\left(x+\frac{1}{3}-\frac{1}{18x+3}\right)\right).
\end{equation*}
\end{thm}
\begin{proof} Since the function $e^{-\psi(x)}$ is strictly decreasing on $(0,\infty)$, in the light of (\ref{e:7}), the only thing we need to show for $x>0$ is the following inequality:
\begin{equation*}
\frac{1}{2}\frac{1}{(x+1)\log(1+1/x)-1}>x+\frac{1}{3}-\frac{1}{18x+3}
\end{equation*}
or equivalently
\begin{equation*}
P(x):=\log(1+1/x)-\frac{1+12x+12x^2}{6x(x+1)(2x+1)}<0.
\end{equation*}
Differentiation gives
\begin{equation*}
P'(x)=\frac{1}{6x^2(1+x)^2(1+2x)^2}>0.
\end{equation*}
Thus, $P$ is strictly increasing on $(0,\infty)$, which implies that $P(x)<\lim_{x\to\infty}P(x)=0$.
\end{proof}
\begin{thm} For all $x>0$, we have
\begin{equation}\label{e:23}
\exp\left(x\psi\left(\frac{x}{2}+1-\frac{x^2}{12+2x}\right)\right)<\Gamma(x+1)<\exp\left(x\psi\left(\frac{x}{2}+1\right)\right).
\end{equation}
\end{thm}
\begin{proof}Since the function $\psi$ is an increasing function, taking into account (\ref{e:8}), it suffices to see

\begin{equation}\label{e:24}
\frac{x}{2}+1-\frac{x^2}{12+2x}<\frac{x}{\log(x+1)}
\end{equation}
or after a little simplification,
\begin{equation}\label{e:25}
p(x):=\log(x+1)-\frac{x^2+6x}{4x+6}>0.
\end{equation}
Differentiating, we get
$$
p'(x)=-\frac{x^2}{(x+1)(3+2x)^2},
$$
so that $p$ is strictly decreasing on $(0,\infty)$. Thus we have $p(x)>\lim_{x\to\infty}p(x)=0$.
\end{proof}
\begin{rem}
The upper bounds in (\ref{e:6}) and (\ref{e:9}) are slightly better than those of (\ref{e:14}) and (\ref{e:16}), but the upper bounds in (\ref{e:14}) and (\ref{e:16}) have the advantage of simplicity.
\end{rem}

\bibliographystyle{amsplain}

\begin{thebibliography}{1}
\bibitem{1} H. Alzer, On some inequalities for the gamma and psi functions, Math. Comp., 66(1997), no. 217, 373-389.
\bibitem{2} H. Alzer, Sharp inequalities for the digamma and polygamma functions, Forum Math. 16 (2004) 181-221
\bibitem{3} H. Alzer and N. Batir, Monotonicity properties of the gamma function, Appl. Math. Letters, 20(2007), 778-781.
\bibitem{4} H. Alzer, J. Wells, Inequalities for the polygamma functions, SIAM J. Math. Anal. 29 (1998) 1459-1466.
\bibitem{5} N. Batir, Bounds for the gamma function, Results Math., 72, 2017, 865-874.
\bibitem{6} L. Gordon, A stochastic approach to the gamma function, Amer. Math. Monthly, 9(101), 1994, 858-865.
\bibitem{7} B.-N. Guo, F. Qi, J.-L. Zhao, Q.-M. Luo, Sharp inequalities for polygamma functions, Math. Slovaca 65 (1) (2015) 103-120
\bibitem{8} B.-N. Guo, F. Qi, Refinements of lower bounds for polygamma functions, Proc. Amer. Math. Soc. 141 (3) (2013) 1007-1015,
\bibitem{9} J. E. Marsden, Basic Complex Analysis, W. H. Freeman and Company, San Fransisco, 1973.
\bibitem{10} M. Minc and L. Sathre, Some inequalities involving $(r!)^{1/r}$, Edinburgh Math. Soc., 14(1964/65), 41-46.
\bibitem{11} C. Mortici, Accurate estimates of the gamma function involving the digamma function, Numer. Funct. Anal. Optimization, 32(4),2011, 469-476.
\bibitem{12} F. Qi, C. Mortici, Some inequalities for the trigamma function in terms of the digamma function, Appl. Math. Comput.
271 (15) (2015) 502-511
\bibitem{13} Zh.-H. Yang, Y.-M. Chu, X.-H. Zhang, Sharp bounds for psi function, Appl. Math. Comput. 268 (2015) 1055-1063,
\bibitem{14} Zh.-H. Yang, Y.-M. Chu, X.-H. Zhang, Necessary and sufficient conditions for functions involving the psi function to be
completely monotonic, J. Inequal. Appl. 2015 (2015) 157
\bibitem{15} Zh.-H. Yang, Sh.-Zh. Zheng, Monotonicity of a mean related to psi and polygamma functions with an application, J. Inequal.
Appl. 2016 (2016) 216
\bibitem{16} Z.-H. Yang and Y.-M. Chu, Asymptotic formulas for gamma function
with applications, Appl. Math. Comput., 2015, 270, 665-680.
 \bibitem{17} Z.-H. Yang, W.-M. Qian, Y.-M. Chu and W. Zhang, On rational
bounds for the gamma function, J. Inequal. Appl., 2017, 2017, Article ID
210, 17 pages.
\bibitem{18} X.-M. Zhang and Y.-M. Chu, A double inequality for gamma function,
J. Inequal. Appl., 2009, Article ID 503782, 7 pages.
\bibitem{19}  T.-H. Zhao, Zh.-H. Yang, Y.-M. Chu, Monotonicity properties of a function involving the psi function with applications,
J. Inequal. Appl. 2015 (2015) 193,
\bibitem{20}T.-H. Zhao, Y.-M. Chu and Y.-P. Jiang, Monotonic and logarithmically
convex properties of a function involving gamma functions, J. Inequal.
Appl., 2009, Article ID 728612, 13 pages.
\bibitem{21} T.-H. Zhao and Y.-M. Chu, A class of logarithmically completely
monotonic functions associated with a gamma function, J. Inequal. Appl.,
2010, Article ID 392431, 11 pages.
\bibitem{22} T.-H. Zhao, Y.-M. Chu and H. Wang, Logarithmically complete
monotonicity properties relating to the gamma function, Abstr. Appl. Anal.,
2011, 2011, Article ID 896483, 13 pages.
\end{thebibliography}

\end{document}